\pdfoutput=1
\documentclass[12pt]{article}
\usepackage{amssymb,amsmath,latexsym}
\usepackage[affil-it]{authblk}
\usepackage{amsthm}
\usepackage[left=2.1cm,top=2cm,right=2.1cm, bottom=2cm,letterpaper]{geometry}
\usepackage{enumitem}
\usepackage{titlesec}
\titlelabel{\thetitle.\quad}

\usepackage{setspace}
\singlespace
\usepackage{hyperref}
\hypersetup{
    colorlinks,
    citecolor=black,
    filecolor=black,
    linkcolor=black,
    urlcolor=black
}

\newtheorem{proposition}{Proposition}[section]

\numberwithin{figure}{section}
\numberwithin{equation}{section}

\begin{document}
\title{Further remarks on the Luo-Hou's ansatz for a self-similar solution to the 3D Euler equations}
\author{Gianmarco Sperone}
\affil{Dipartimento di Matematica\\
       Politecnico di Milano\\
       Piazza Leonardo da Vinci 32 - 20133 Milan (Italy)\\
       E-mail: gianmarcosilvio.sperone@polimi.it}

\date{}       
\maketitle
\abstract{It is shown that the self-similar ansatz proposed by T. Hou and G. Luo to describe a singular solution of the 3D axisymmetric Euler equations leads, without assuming any asymptotic condition on the self-similar profiles, to an over-determined system of partial differential equations that produces two families of solutions: a class of trivial solutions in which the vorticity field is identically zero, and a family of solutions that blow-up immediately, where the vorticity field is governed by a stationary regime. In any case, the analytical properties of these solutions are not consistent with the numerical observations reported by T. Hou and G. Luo. Therefore, this result is a refinement of the previous work published by D. Chae and T.-P. Tsai on this matter, where the authors found the trivial class of solutions under a certain decay condition of the blow-up profiles.}

\section{Introduction}
The motion of a perfect, incompressible and homogeneous fluid through a region $\Omega \subseteq \mathbb{R}^{3}$, in the absence of external forcing, can be described by the celebrated Euler equations:
\begin{equation} \label{eulerinc}
\left\lbrace 
		\begin{aligned}
		 & \dfrac{\partial \textbf{u}}{\partial t} + \textbf{u} \cdot \nabla \textbf{u} = -     \nabla p
		 \\ & \nabla \cdot \textbf{u}=0
		\end{aligned}\ \ \ \ \text{in} \ \ \Omega \times (0, \infty),  
\right.	
\end{equation}
where $\textbf{u} : \Omega \times [0,\infty) \longrightarrow \mathbb{R}^{3}$ is the velocity field and $p : \Omega \times [0,\infty) \longrightarrow \mathbb{R}$ is the pressure field. One of the greatest unsolved problems in the mathematical theory of fluid mechanics is to determine whether regular solutions of the system \eqref{eulerinc} can develop finite-time singularities. An instructive list containing some of the numerical studies that support or refute this possibility may be found in \cite{gibbon}, where one could add the most recent numerical experiment designed by T. Hou and G. Luo (see \cite{luo}, or the resumed version \cite{luo_res}), in which they identified a class of potentially singular solutions to the 3D axisymmetric Euler equations in a radially bounded, axially periodic cylinder. Precisely, they found that the point of the potential singularity, which corresponds to the point of the maximum vorticity, is always located at the intersection of the solid boundary $r = 1$ and the symmetry plane $z = 0$ (in cylindrical coordinates), and the estimated singularity time is $T \approx 0$.$0035056$ [s]. Finally, a local analysis near the point of singularity also suggested the existence of a self-similar blow-up in the meridian plane.

\section{The Euler equations for 3D axisymmetric flows}
Let us consider a cylindrical coordinate system $(r,\theta,z) \in [0,\infty) \times [0,2\pi] \times \mathbb{R}$, in which any spatial point will be denoted by $\xi = r\widehat{r} + z\widehat{k}$. with $r \geq 0$, $z \in \mathbb{R}$ and $\left\lbrace \widehat{r}, \widehat{\theta}, \widehat{k}\right\rbrace \subset \mathbb{R}^{3}$ the orthonormal basis in this geometry. Given a fixed number $L > 0$, we will consider the following domain:
$$
\Sigma_{L} = \left\lbrace \xi \in \mathbb{R}^{3} \ \vert \ 0 < r < 1, \ 0 < z < L \right\rbrace.
$$ 
\par
By definition, an axisymmetric velocity field has the following representation in cylindrical coordinates:
$$
\textbf{u}(\xi,t) = u_{r}(r,z,t)\widehat{r} + u_{\theta}(r,z,t)\widehat{\theta} + u_{z}(r,z,t)\widehat{k},
$$ 
for all $(\xi,t) \in \overline{\Sigma_{L}} \times [0,\infty)$. The vorticity field (that is, the curl of the velocity field) admits a similar representation:
$$
\omega(\xi,t) = \omega_{r}(r,z,t)\widehat{r} + \omega_{\theta}(r,z,t)\widehat{\theta} + \omega_{z}(r,z,t)\widehat{k},
$$
for every $(\xi,t) \in \Sigma_{L} \times [0,\infty)$. Now, the incompressibility condition implies the existence of an axisymmetric stream function:
$$
\psi(\xi,t) = \psi_{r}(r,z,t)\widehat{r} + \psi_{\theta}(r,z,t)\widehat{\theta} + \psi_{z}(r,z,t)\widehat{k}, \ \ \forall (\xi,t) \in \overline{\Sigma_{L}} \times [0,\infty),
$$
that satisfies the following identities:
$$
\left\lbrace 
		\begin{aligned}
		 & \textbf{u}(\xi,t) = \text{rot}(\psi)(\xi,t)
		 \\ & \omega(\xi,t) = -\Delta\psi(\xi,t),
		\end{aligned}  
\right.	
$$
for all $(\xi,t) \in \Sigma_{L} \times [0,\infty)$. Then, in order to find an axisymmetric solution of \eqref{eulerinc}, the incompressible 3D Euler equations may be alternatively formulated in terms of the angular components $(u_{\theta},\omega_{\theta},\psi_{\theta})$ as follows (for a complete construction of this formulation, see \cite{majda}, chapter  N\textsuperscript{o}2):
\begin{equation} \label{euler_alternativa1}
\left\lbrace 
		\begin{aligned}
		 & \dfrac{\partial u_{\theta}}{\partial t} + u_{r}\dfrac{\partial u_{\theta}}{\partial r} + u_{z}\dfrac{\partial u_{\theta}}{\partial z} = -\dfrac{u_{r} u_{\theta}}{r}
		 \\ & \dfrac{\partial \omega_{\theta}}{\partial t} + u_{r}\dfrac{\partial \omega_{\theta}}{\partial r} + u_{z}\dfrac{\partial \omega_{\theta}}{\partial z} = \dfrac{2 u_{\theta}}{r}\dfrac{\partial u_{\theta}}{\partial z} + \dfrac{u_{r} \omega_{\theta}}{r}
		 \\ & - \left( \Delta - \dfrac{1}{r^{2}} \right) \psi_{\theta} = \omega_{\theta} 
		\end{aligned}  
\right. \ \ \ \ \text{in} \ \ \Sigma_{L} \times (0, \infty).
\end{equation}
\par
The original system \eqref{euler_alternativa1} presents an artifical singularity at $r = 0$, which is inconvenient to work with numerically. Therefore, the authors of \cite{luo} employed the following transformations:
$$
u_{1}(r,z,t) \doteq \dfrac{1}{r}u_{\theta}(r,z,t), \ \  \omega_{1}(r,z,t) \doteq \dfrac{1}{r}\omega_{\theta}(r,z,t), \ \ \psi_{1}(r,z,t) \doteq \dfrac{1}{r}\psi_{\theta}(r,z,t), 	
$$
for every $(\xi,t) \in \Sigma_{L} \times [0,\infty)$, which allow us to write the system \eqref{euler_alternativa1} in terms of $(u_{1},\omega_{1},\psi_{1})$ as:

\begin{equation} \label{euler_alternativa2}
\left\lbrace 
		\begin{aligned}
		 & \dfrac{\partial u_{1}}{\partial t} + u_{r}\dfrac{\partial u_{1}}{\partial r} + u_{z}\dfrac{\partial u_{1}}{\partial z} = 2u_{1}\dfrac{\partial \psi_{1}}{\partial z}
		 \\ & \dfrac{\partial \omega_{1}}{\partial t} + u_{r}\dfrac{\partial \omega_{1}}{\partial r} + u_{z}\dfrac{\partial \omega_{1}}{\partial z} = \dfrac{\partial}{\partial z}(u_{1}^{2})
		 \\ & -\left[ \left(\dfrac{\partial}{\partial r} \right)^{2} + \dfrac{3}{r} \dfrac{\partial}{\partial r} + \left(\dfrac{\partial}{\partial z} \right)^{2} \right] \psi_{1} = \omega_{1}
		\end{aligned}  
\right. \ \ \ \ \text{in} \ \ \Sigma_{L} \times (0, \infty).
\end{equation}
\par
The radial and axial components of the velocity field can be recovered as follows:
\begin{equation} \label{euler_alternativa3}
\left\lbrace 
		\begin{aligned}
		 & u_{r}(r,z,t) = -r\dfrac{\partial \psi_{1}}{\partial z}(r,z,t)
		 \\ & u_{z}(r,z,t) = 2\psi_{1}(r,z,t) + r\dfrac{\partial \psi_{1}}{\partial r}(r,z,t),
		\end{aligned}  
\right.	
\end{equation}
for all $(\xi,t) \in \Sigma_{L} \times [0,\infty)$, after which the incompressibility condition:
$$
\dfrac{1}{r}\dfrac{\partial}{\partial r}\left(r u_{r}\right) + \dfrac{\partial u_{z}}{\partial z} = 0, \ \  \text{in} \ \ \Sigma_{L} \times [0, \infty),
$$
is automatically satisfied.

\section{Analysis of Luo-Hou's self-similar ansatz}
\par
The authors of \cite{luo} numerically solved the system of partial differential equations \eqref{euler_alternativa2} on the cylinder $\Sigma_{L}$, under a suitable periodic condition along the axial direction and a no-flow boundary condition on the solid wall $r = 1$. Thus, according to the authors of \cite{chae4}, it is reasonable to consider the equations \eqref{euler_alternativa2} in the following space-time domain:
$$
\left\lbrace  (r,z) \in \mathbb{R}^{2} \ \vert \ 0 < r < 1, \  -\infty < z < \infty \right\rbrace \times [0,T),
$$
with $T > 0$ being a possible explosion time. 
\par
As mentioned earlier, T. Hou and G. Luo found rigorous evidence that suggested the formation of a ring-singularity on the solid boundary (due to the rotational symmetry of the flow), at a particular instant of time $T > 0$. In order to describe analytically a possible blow-up scenario at the space-time point $(r,z,t) = (1,0,T)$, T. Hou and G. Luo proposed the following local self-similar ansatz for the solutions of \eqref{euler_alternativa2}:
\begin{equation} \label{self_similar0}
\left\lbrace 
		\begin{aligned}
		 & u_{1}(r,z,t) = (T - t)^{\gamma_{u}} U\left( \dfrac{r - 1}{[T - t]^{\gamma_{l}}},\dfrac{z}{[T - t]^{\gamma_{l}}}\right) 
		 \\ & \omega_{1}(r,z,t) = (T - t)^{\gamma_{\omega}} \Omega\left( \dfrac{r - 1}{[T - t]^{\gamma_{l}}},\dfrac{z}{[T - t]^{\gamma_{l}}}\right) 
		 \\ & \psi_{1}(r,z,t) = (T - t)^{\gamma_{\psi}} \Psi\left( \dfrac{r - 1}{[T - t]^{\gamma_{l}}},\dfrac{z}{[T - t]^{\gamma_{l}}}\right), 
		\end{aligned}  
\right.	
\end{equation}
for every $r \approx 1$, $z \approx 0$ and $t \approx T$; in other words, the ansatz is valid on a neighborhood of the unit circle, for all time sufficiently close to the expected blow-up time. In this case, $(U, \Omega, \Psi)$ represent the self-similar profiles (which describe the development and evolution of the self-similar singularity) and $\gamma_{u}, \gamma_{\omega}, \gamma_{\psi}, \gamma_{l} \in \mathbb{R}$ are scaling exponents. As proved in \cite{luo} (section 4.7), upon substitution of \eqref{self_similar0} into \eqref{euler_alternativa2}, one can easily show that these scaling exponents are organized as a one-parameter family:

$$
\gamma_{u} = -1 + \dfrac{1}{2}\gamma, \ \ \ \gamma_{\omega} = -1, \ \ \ \gamma_{\psi} = -1 + 2\gamma,
$$
with $\gamma \doteq \gamma_{l} \approx 2$,$9133$. At this point, it is convenient to recall the following important result due to P. Constantin (see \cite{constantin2}): in view of the conservation of energy, a necessary condition for the existence of a finite-time blow-up is that $\gamma \geq 2 / 5$.
\par
For our analysis in this section, and imitating the methodology employed by D. Chae and T.-P. Tsai in \cite{chae4}, we will simply assume that $\gamma > 0$ and that the self-similar ansatz \eqref{self_similar0} is valid in the following space-time region:
\begin{equation} \label{self_similar1}
\mathcal{C}(\delta,T) \doteq \left\lbrace  (r,z,t) \in \mathbb{R}^{3} \ \vert \ 1 - \delta < r < 1, \ \ - \delta < z < \delta, \ \ T - \delta < t < T \right\rbrace,
\end{equation}
for some number $0 < \delta \ll 1$. By introducing the change of variables:
$$
R \doteq \dfrac{r - 1}{(T - t)^{\gamma}} \ \ , \ \ Z \doteq \dfrac{z}{(T - t)^{\gamma}},
$$
for all $(r,z,t) \in \mathcal{C}(\delta,T)$, we deduce that the self-similar profiles $(U, \Omega, \Psi)$ must be defined on the closure of the left-hand plane:
\begin{equation} \label{self_similar2}
\mathcal{D} \doteq \left\lbrace  (R,Z) \in \mathbb{R}^{2} \ \vert \ -\infty < R < 0, \ \ -\infty < Z < \infty \right\rbrace.
\end{equation}
The main result of this article is contained in the next proposition:
\begin{proposition} \label{result}
Let \textnormal{$(u_{1},\omega_{1},\psi_{1})$} be a classical solution of the system \eqref{euler_alternativa2} having the following representation:
\textnormal{
\begin{equation} \label{self_similar3}
\left\lbrace 
		\begin{aligned}
		 & u_{1}(r,z,t) = (T - t)^{-1 + \frac{\gamma}{2}} U(R,Z) 
		 \\ & \omega_{1}(r,z,t) = (T - t)^{-1} \Omega(R,Z) 
		 \\ & \psi_{1}(r,z,t) = (T - t)^{{-1 + 2\gamma}} \Psi(R,Z), 
		\end{aligned}  
\right.
\end{equation}
}for every \textnormal{$(r,z,t) \in \mathcal{C}(\delta,T)$}, and for some fixed parameter \textnormal{$\gamma > 0$}. Then, the trio \textnormal{$(u_{1},\omega_{1},\psi_{1})$} leads to an over-determined system of partial differential equations that produces two families of solutions:
\begin{itemize}[leftmargin=*]
\item[1.] A class of trivial solutions given by:
\textnormal{
\begin{equation} \label{solution1}
\left\lbrace 
		\begin{aligned}
		 & u_{1}(r,z,t) = 0
		 \\ & \omega_{1}(r,z,t) = 0
		 \\ & \psi_{1}(r,z,t) = b(T - t)^{-1 + \gamma}z + c(T - t)^{-1 + 2\gamma}, 
		\end{aligned}  
\right.	
\end{equation}
}
for all $(r,z,t) \in  \mathcal{C}(\delta,T)$, with $ b, c \in \mathbb{R}$.

\item[2.] A class of solutions that blow-up instantaneously, corresponding to the expression:
\textnormal{
\begin{equation} \label{solution2}
\left\lbrace 
		\begin{aligned}
		 & u_{1}(r,z,t) = \kappa(1 - r)^{\frac{\gamma - 2}{2\gamma}}
		 \\ & \omega_{1}(r,z,t) = 0
		 \\ & \psi_{1}(r,z,t) = c(T - t)^{-1 + 2\gamma}, 
		\end{aligned}  
\right.	
\end{equation}
}
for all $(r,z,t) \in  \mathcal{C}(\delta,T)$, with $\kappa, c \in \mathbb{R}$.
\end{itemize}
\end{proposition}

\begin{proof}
We begin by writing the radial and axial components of the velocity field in terms of the self-similar profiles $(U, \Omega, \Psi)$, according to the expressions \eqref{euler_alternativa3}:

\begin{equation} \label{calc1}
\left\lbrace 
		\begin{aligned}
        & u_{r}(r,z,t) = -[1 + R(T-t)^{\gamma}](T-t)^{-1 + \gamma}\dfrac{\partial \Psi}{\partial Z}(R,Z),
        \\ & u_{z}(r,z,t) = 2(T-t)^{-1 + 2\gamma}\Psi(R,Z) + [1 + R(T-t)^{\gamma}](T-t)^{-1 + \gamma}\dfrac{\partial \Psi}{\partial R}(R,Z),
       \end{aligned}  
\right.	
\end{equation}
for every $(r,z,t) \in  \mathcal{C}(\delta,T)$. Let us now introduce the following notation for the differential operators and vectors that are going to be employed:
$$
\nabla = \left( \dfrac{\partial}{\partial R}, \dfrac{\partial}{\partial Z}\right) , \ \ \ \nabla^{\perp} = \left( -\dfrac{\partial}{\partial Z}, \dfrac{\partial}{\partial R}\right) , \ \ \ \Delta = \dfrac{\partial^{2}}{\partial R^{2}} + \dfrac{\partial^{2}}{\partial Z^{2}}, \ \ \ Y = (R,Z).
$$
$\bullet$ By substituting \eqref{self_similar3} and \eqref{calc1} into the first equation of \eqref{euler_alternativa2}, we obtain:
\begin{equation} \label{calc2}
\begin{aligned}
& \left(1 - \dfrac{\gamma}{2} \right) (T-t)^{-2 + \frac{\gamma}{2}}U + \gamma(T-t)^{-2 + \frac{\gamma}{2}}\left(R\dfrac{\partial U}{\partial R} + Z \dfrac{\partial U}{\partial Z}\right) \\ & - [1 + R(T-t)^{\gamma}](T-t)^{-2 + \frac{\gamma}{2}}\dfrac{\partial \Psi}{\partial Z} \dfrac{\partial U}{\partial R}
\\ & + \left\lbrace 2(T-t)^{-1 + 2\gamma}\Psi + [1 + R(T-t)^{\gamma}](T-t)^{-1 + \gamma}\dfrac{\partial \Psi}{\partial R} \right\rbrace (T-t)^{-1 - \frac{\gamma}{2}}\dfrac{\partial U}{\partial Z}
\\ & = 2(T-t)^{-2 + \frac{3}{2}\gamma}U\dfrac{\partial \Psi}{\partial Z},
\end{aligned}  
\end{equation}
for all $(R,Z) \in \mathcal{D}$ and $t \in [T - \delta,T)$. Multiplying \eqref{calc2} by $(T-t)^{2 - \frac{\gamma}{2}}$, and then taking the limit as $t \nearrow T$ in that expression, we get (since $\gamma > 0$):
$$
\left(1 - \dfrac{\gamma}{2} \right)U + \gamma Y \cdot \nabla U + \nabla^{\perp} \Psi \cdot \nabla U = 0,
$$
for every $(R,Z) \in \mathcal{D}$.\\
\\
\noindent
$\bullet$ By substituting \eqref{self_similar3} and \eqref{calc1} into the second equation of \eqref{euler_alternativa2}, we obtain:
\begin{equation} \label{calc3}
\begin{aligned}
& (T - t)^{-2}\Omega + \gamma(T-t)^{-2}\left(R\dfrac{\partial \Omega}{\partial R} + Z\dfrac{\partial \Omega}{\partial Z}\right) - [1 + R(T-t)^{\gamma}](T-t)^{-2}\dfrac{\partial \Psi}{\partial Z}\dfrac{\partial \Omega}{\partial R}
\\ & + \left\lbrace 2(T-t)^{-1 + 2\gamma}\Psi + [1 + R(T-t)^{\gamma}](T-t)^{-1 + \gamma}\dfrac{\partial \Psi}{\partial R} \right\rbrace (T-t)^{-1 - \gamma}\dfrac{\partial \Omega}{\partial Z}
\\ & = (T-t)^{-2}\dfrac{\partial}{\partial Z}(U^{2}),
\end{aligned}  
\end{equation}
for all $(R,Z) \in \mathcal{D}$ and $t \in [T - \delta,T)$. Multiplying \eqref{calc3} by $(T-t)^{2}$, and then taking the limit as $t \nearrow T$ in that expression, we get: 
$$
\Omega + \gamma Y \cdot \nabla \Omega + \nabla^{\perp} \Psi \cdot \nabla \Omega = \dfrac{\partial}{\partial Z}(U^{2}),
$$
for every $(R,Z) \in \mathcal{D}$.\\
\\
\noindent
$\bullet$ By substituting \eqref{self_similar3} into the third equation of \eqref{euler_alternativa2}, we obtain:
\begin{equation} \label{calc4}
-(T-t)^{-1}\left( \dfrac{\partial^{2} \Psi}{\partial R^{2}} + \dfrac{\partial^{2} \Psi}{\partial Z^{2}}\right)  - \dfrac{3(T-t)^{-1 + \gamma}}{1 + R(T-t)^{\gamma}}\dfrac{\partial \Psi}{\partial R} = (T-t)^{-1}\Omega,
\end{equation}
for all $(R,Z) \in \mathcal{D}$ and $t \in [T - \delta,T)$. Multiplying \eqref{calc4} by $(T-t)$, and then taking the limit as $t \nearrow T$ in that expression, we get:
$$
- \Delta \Psi = \Omega,
$$
for every $(R,Z) \in \mathcal{D}$.
\par
Therefore, the first group of dominant equations for the self-similar profiles $(U, \Omega, \Psi)$ is given by:
\begin{equation} \label{dom215}
\left(1 - \dfrac{\gamma}{2} \right)U + \gamma Y \cdot \nabla U + \nabla^{\perp} \Psi \cdot \nabla U = 0
\end{equation}
\begin{equation} \label{dom216}
\Omega + \gamma Y \cdot \nabla \Omega + \nabla^{\perp} \Psi \cdot \nabla \Omega = \dfrac{\partial}{\partial Z}(U^{2})
\end{equation}
\begin{equation} \label{dom217}
- \Delta \Psi = \Omega,
\end{equation}
for every $(R,Z) \in \mathcal{D}$. As pointed out by D. Chae and T.-P. Tsai in \cite{chae4}, the interesting fact is that we can actually use identities \eqref{dom215} - \eqref{dom216} - \eqref{dom217} to obtain another group of dominant equations for the blow-up profiles $(U, \Omega, \Psi)$. The derivation goes as follows:\\
\\
\noindent
$\bullet$ Multiplying \eqref{calc2} by $(T-t)^{2 - \frac{3}{2}\gamma}$ and reordering terms, we get:
$$
(T - t)^{-\gamma} \left\lbrace  \left(1 - \dfrac{\gamma}{2} \right) U + \gamma (Y \cdot \nabla U) + \nabla^{\perp} \Psi \cdot \nabla U \right\rbrace  + 2\Psi \dfrac{\partial U}{\partial Z} + R(\nabla^{\perp} \Psi \cdot \nabla U) = 2U\dfrac{\partial \Psi}{\partial Z}
$$
$$
\Leftrightarrow \ \ R(\nabla^{\perp} \Psi \cdot \nabla U) + 2\Psi \dfrac{\partial U}{\partial Z} = 2U\dfrac{\partial \Psi}{\partial Z},
$$
for all $(R,Z) \in \mathcal{D}$, thanks to identity \eqref{dom215}.\\
\\
\noindent
$\bullet$ Multiplying \eqref{calc3} by $(T-t)^{2 - \gamma}$ and reordering terms, we get:
$$
(T - t)^{-\gamma} \left\lbrace \Omega + \gamma (Y \cdot \nabla \Omega) + \nabla^{\perp} \Psi \cdot \nabla \Omega - \dfrac{\partial}{\partial Z}(U^{2}) \right\rbrace + 2\Psi \dfrac{\partial \Omega}{\partial Z} + R(\nabla^{\perp} \Psi \cdot \nabla \Omega) = 0
$$
$$
\Leftrightarrow \ \ R(\nabla^{\perp} \Psi \cdot \nabla \Omega) + 2\Psi \dfrac{\partial \Omega}{\partial Z} = 0,
$$
for all $(R,Z) \in \mathcal{D}$, thanks to identity \eqref{dom216}.\\
\\
\noindent
$\bullet$ Multiplying \eqref{calc4} by $(T-t)^{1 - \gamma}$ and reordering terms, we get:
$$
- \dfrac{3}{1 + R(T-t)^{\gamma}}\dfrac{\partial \Psi}{\partial R} = (T-t)^{-\gamma}(\Delta \Psi + \Omega)
$$
$$
\Leftrightarrow \ \ \dfrac{\partial \Psi}{\partial R} = 0,
$$
for all $(R,Z) \in \mathcal{D}$, thanks to identity \eqref{dom217}.
\par
Thus, the second group of dominant equations for the self-similar profiles $(U, \Omega, \Psi)$ is:
\begin{equation} \label{dom218}
R(\nabla^{\perp} \Psi \cdot \nabla U) + 2\Psi \dfrac{\partial U}{\partial Z} = 2U\dfrac{\partial \Psi}{\partial Z}
\end{equation}
\begin{equation} \label{dom219}
R(\nabla^{\perp} \Psi \cdot \nabla \Omega) + 2\Psi \dfrac{\partial \Omega}{\partial Z} = 0
\end{equation}
\begin{equation} \label{dom220}
\dfrac{\partial \Psi}{\partial R} = 0,
\end{equation}
for all $(R,Z) \in \mathcal{D}$, so identities \eqref{dom215} to \eqref{dom220} constitute an over-determined system of partial differential equations for the blow-up profiles $(U, \Omega, \Psi)$.
\par
From equality \eqref{dom220} we immediately deduce that the profile $\Psi$ only depends on the variable $Z \in \mathbb{R}$. According to \eqref{dom217}, the same applies for profile $\Omega$, together with the identity:
\begin{equation} \label{dom221}
-\Psi ''(Z) = \Omega(Z), \ \forall Z \in \mathbb{R}.
\end{equation}
With this information, equation \eqref{dom219} becomes:
\begin{equation} \label{dom222}
\Psi(Z) \Omega'(Z) = 0, \ \forall Z \in \mathbb{R}.
\end{equation}
Now, suppose that $\Omega \in \mathcal{C}^{1}(\mathbb{R}; \mathbb{R})$ and that its derivative $\Omega'$ is not identically zero on $\mathbb{R}$. Then, without loss of generality, we may choose $Z_{0} \in \mathbb{R}$ such that $\Omega'(Z_{0}) > 0$. By continuity, there exists $\eta > 0$ such that $\Omega'(Z) > 0$, for all $Z \in (Z_{0} - \eta, Z_{0} + \eta)$. According to \eqref{dom222}, we obtain $\Psi(Z) = 0$, $\forall Z \in B(Z_{0},\eta)$, and this also implies that $\Omega(Z) =0$, for every $Z \in (Z_{0} - \eta, Z_{0} + \eta)$, thanks to \eqref{dom221}. In particular, we have $\Omega'(Z)=0$, $\forall Z \in (Z_{0} - \eta, Z_{0} + \eta)$, contradicting the initial hypothesis. Thus, $\Omega'$ must be identically zero on $\mathbb{R}$, or well, there exists $a \in \mathbb{R}$ such that $\Omega(Z) = a$, $\forall Z \in \mathbb{R}$. But if the profile $\Omega$ is constant on $\mathbb{R}$, equation \eqref{dom221} allows us to conclude that $\Psi$ must be a quadratic function on $\mathbb{R}$:
\begin{equation} \label{dom223}
\Psi(Z) = -\dfrac{a}{2}Z^{2} + bZ + c, \ \forall Z \in \mathbb{R},
\end{equation}
for some constants $b, c \in \mathbb{R}$.
\par
Let us now suppose that $a \neq 0$. With all the previous information, identity \eqref{dom216} reduces to:
$$
\dfrac{\partial}{\partial Z}(U^{2}) = a, \ \ \forall (R,Z) \in \mathcal{D}.
$$
Therefore, there is a differentiable function $f : (-\infty, 0] \longrightarrow \mathbb{R}$ such that:
$$
U(R,Z)^{2} = f(R) + aZ, \ \ \forall (R,Z) \in \mathcal{D}.
$$
Differentiating with respect to $R<0$ and $Z \in \mathbb{R}$, we get:
\begin{equation} \label{dom226}
\left\lbrace 
		\begin{aligned}
        & 2U(R,Z)\dfrac{\partial U}{\partial R}(R,Z) = f'(R)
        \\ & 2U(R,Z)\dfrac{\partial U}{\partial Z}(R,Z) = a,
       \end{aligned}  
\right.	
\end{equation}
for every $(R,Z) \in \mathcal{D}$. On the other hand, equation \eqref{dom215} is simplified to:
$$
\left(1 - \dfrac{\gamma}{2} \right)U(R,Z) + \gamma R \dfrac{\partial U}{\partial R}(R,Z) + \gamma Z \dfrac{\partial U}{\partial Z}(R,Z) - \Psi'(Z)\dfrac{\partial U}{\partial R}(R,Z) = 0, 
$$
for all $(R,Z) \in \mathcal{D}$. Multiplying this identity by $2U(R,Z)$, with $(R,Z) \in \mathcal{D}$, and then using \eqref{dom226}, we deduce that:
\begin{equation} \label{dom228}
(2-\gamma)[f(R) + aZ] + [\gamma R  + aZ - b]f'(R) + a \gamma Z = 0, \ \ \forall (R,Z) \in \mathcal{D}.
\end{equation}
Dividing both sides of equation \eqref{dom228} by $Z > 0$, and then taking the limit as $Z \longrightarrow +\infty$, we may conclude that:
$$
(2-\gamma)a + af'(R) + a\gamma = 0, \ \ \forall (R,Z) \in \mathcal{D},
$$
or equivalently, $f'(R) = -2$, $\forall R < 0$, since $a \neq 0$. As a consequence, there exists a constant $\kappa \in \mathbb{R}$ such that:

\begin{equation} \label{dom230}
U(R,Z)^{2} = -2R + aZ + \kappa, \ \ \forall (R,Z) \in \mathcal{D}.
\end{equation}
With all this new information, identity \eqref{dom218} becomes:
\begin{equation} \label{dom231}
-R \Psi'(Z) \dfrac{\partial U}{\partial R}(R,Z) + 2 \Psi(Z) \dfrac{\partial U}{\partial Z}(R,Z) = 2 \Psi'(Z) U(R,Z), \ \ \forall (R,Z) \in \mathcal{D}.
\end{equation}
Like before, multiplying this equation by $2U(R,Z)$, with $(R,Z) \in \mathcal{D}$, and then using \eqref{dom226} and \eqref{dom230}, we get the following:
\begin{equation} \label{dom232}
2R\Psi'(Z) + 2a\Psi(Z) = 4 \Psi'(Z)(-2R + aZ + \kappa), \ \ \forall (R,Z) \in \mathcal{D}.
\end{equation}
Dividing both sides of equation \eqref{dom232} by $R < 0$, and then taking the limits as $R \longrightarrow -\infty$, one concludes that $\Psi'(Z) = 0$, for every $Z \in \mathbb{R}$. Given the quadratic form of profile $\Psi$, this necessarily implies that $a = b = 0$, contradicting the initial hypothesis.
\par
As a consequence, we must have $a = 0$, and so:
\begin{equation} \label{dom233}
\left\lbrace 
		\begin{aligned}
        & \Omega(Z) = 0
        \\ & \Psi(Z) = bZ + c,
       \end{aligned}  
\right.	
\end{equation}
for all $Z \in \mathbb{R}$. We still haven't employed the information contained in equations  \eqref{dom215}, \eqref{dom216} and \eqref{dom218}. First of all, \eqref{dom216} simplifies to:
$$
\dfrac{\partial}{\partial Z}(U^{2}) = 0, \ \ \forall (R,Z) \in \mathcal{D}.
$$
This indicates that the function $U^{2}$ does not depend on $Z$, and accordingly, profile $U$ will only depend on the variable $R \leq 0$. Identities \eqref{dom215} and \eqref{dom218} then become:
\begin{equation} \label{dom235}
\left\lbrace 
		\begin{aligned}
        & \left(1 - \dfrac{\gamma}{2} \right)U(R) + \gamma R U'(R) - \Psi'(Z)U'(R)  = 0
        \\ & -R\Psi'(Z)U'(R) = 2U(R)\Psi'(Z),
       \end{aligned}  
\right.	
\end{equation}
for every $R < 0$ and $Z \in \mathbb{R}$. Upon substitution of \eqref{dom233} into \eqref{dom235}, we obtain the following ordinary differential equations for profile $U$:
\begin{equation} \label{dom236}
\left\lbrace 
		\begin{aligned}
        & \left(1 - \dfrac{\gamma}{2} \right)U(R) + (\gamma R - b) U'(R) = 0
        \\ & -bRU'(R) = 2bU(R),
       \end{aligned}  
\right.	
\end{equation}
for all $R < 0$. Our analysis must now distinguish the two subsequent cases:\\
\\
$\bullet$ \textbf{Case (A):} if $b \neq 0$, from the second equality of \eqref{dom236} we may deduce that:
$$
U(R) = -\dfrac{R}{2} U'(R), \ \ \forall R < 0.
$$  
By replacing this into the first equation of \eqref{dom236}, we get:
$$
(\gamma^{*}R - b)U'(R)=0, \ \ \forall R < 0,
$$
with $\gamma^{*} \doteq \dfrac{5\gamma - 2}{4}$. That is, $U'(R) = 0$, for every $R < 0$ such that $R \neq \dfrac{b}{\gamma^{*}}$ (if $\gamma = 2/5$, we directly conclude that $U'(R) = 0$, $\forall R < 0$). If we assume, let's say, that $U \in \mathcal{C}^{1}((-\infty,0]; \mathbb{R})$, we immediately deduce that $U'(R) = 0$, $\forall R < 0$. By continuity, we will also have $U'(R) = 0$, $\forall R \leq 0$. Therefore, in case (A) we have obtained the following family of exact solutions for the trio $(U, \Omega, \Psi)$:
$$
\left\lbrace 
		\begin{aligned}
		 & U(R,Z) = 0
		 \\ & \Omega(R,Z) = 0
		 \\ & \Psi(R,Z) = bZ + c, 
		\end{aligned}  
\right.	
$$
for every $(R,Z) \in \mathcal{D}$, with $b, c \in \mathbb{R}$. In terms of the original functions, the new-found solution is:
\begin{equation} \label{solutionA}
\left\lbrace 
		\begin{aligned}
		 & u_{1}(r,z,t) = 0
		 \\ & \omega_{1}(r,z,t) = 0
		 \\ & \psi_{1}(r,z,t) = b(T - t)^{-1 + \gamma}z + c(T - t)^{-1 + 2\gamma}, 
		\end{aligned}  
\right.	
\end{equation}
for all $(r,z,t) \in  \mathcal{C}(\delta,T)$, with $ b, c \in \mathbb{R}$.\\
\\
\noindent
$\bullet$ \textbf{Case (B):} if $b = 0$, we get a unique ordinary differential equation for profile $U$:
$$
\left(1 - \dfrac{\gamma}{2} \right)U(R) + \gamma R U'(R)=0, \ \ \forall R < 0,
$$
which has a general solution given by the following formula:
$$
U(R) = \kappa \text{exp}\left(- \int\dfrac{1}{\gamma R}\left\lbrace 1 -\dfrac{\gamma}{2} \right\rbrace dR \right) , \ \ \forall R \leq 0,
$$
with $\kappa \in \mathbb{R}$ being a parameter. After a simple manipulation, we get $U(R) = \kappa (-R)^{\frac{1}{2} - \frac{1}{\gamma}}$, for $R \leq 0$ and $\kappa \in \mathbb{R}$. Therefore, in case (B) we have obtained the following family of exact solutions for the trio $(U, \Omega, \Psi)$:
$$
\left\lbrace 
		\begin{aligned}
		 & U(R,Z) = \kappa (-R)^{\frac{1}{2} - \frac{1}{\gamma}}
		 \\ & \Omega(R,Z) = 0
		 \\ & \Psi(R,Z) = c, 
		\end{aligned}  
\right.	
$$
for every $(R,Z) \in \mathcal{D}$, with $\kappa, c \in \mathbb{R}$. In terms of the original functions, the new-found solution is:
\begin{equation} \label{solutionB}
\left\lbrace 
		\begin{aligned}
		 & u_{1}(r,z,t) = \kappa(1 - r)^{\frac{\gamma - 2}{2\gamma}}
		 \\ & \omega_{1}(r,z,t) = 0
		 \\ & \psi_{1}(r,z,t) = c(T - t)^{-1 + 2\gamma}, 
		\end{aligned}  
\right.	
\end{equation}
for all $(r,z,t) \in  \mathcal{C}(\delta,T)$, with $\kappa, c \in \mathbb{R}$.
\end{proof}
We will conclude this section by detailing some of the analytical properties of solutions \eqref{solution1} and \eqref{solution2} that are not consistent with the numerical observations reported by T. Hou and G. Luo in \cite{luo}. 
\par
Recall that the famous Beale - Kato - Majda criterion (BKM criterion; see \cite{beale}, or the generalized version in \cite{ferrari}) states that a smooth solution of the 3D Euler equations blows up at time $T > 0$ if and only if:
\begin{equation} \label{bkm}
\int_{0}^{T} \Vert \omega(\cdot,t) \Vert_{L^{\infty}(\Phi; \mathbb{R}^{3})} \,dt = \infty,
\end{equation}
where $\Phi \subset \mathbb{R}^{3}$ is the spatial domain (that may be bounded or unbounded) in which the 3D Euler equations are being solved. In the case of the solutions associated with the Luo-Hou's self-similar ansatz we have $\Phi = \left\lbrace  (r,z) \in \mathbb{R}^{2} \ \vert \ 0 < r < 1, \  -\infty < z < \infty \right\rbrace$, even though we shall focus our attention in the self-similar region \eqref{self_similar1}, whose spatial section will be denoted as $\Lambda \doteq \left\lbrace  (r,z) \in \mathbb{R}^{2} \ \vert \ 1 - \delta < r < 1, \  -\delta < z < \delta \right\rbrace$. 
\par
The BKM criterion is the basic mathematical tool employed by T. Hou and G. Luo in order to assess the likelihood of a finite-time singularity. Indeed, in \cite{luo} (section $4$.$4$) they describe the numerical procedure used to prove that the simulated solution actually satisfies equality \eqref{bkm}. We will now use this criterion to explain why formulae \eqref{solution1} - \eqref{solution2} do not represent the self-similar regime of smooth solutions of the 3D axisymmetric Euler equations that develop a finite-time singularity.\\
\par
$\bullet$ \textbf{Case (A):} in cylindrical coordinates, the velocity and vorticity fields associated with the class of solutions \eqref{solution1} take the form:
\begin{equation} \label{solution a1}
\left\lbrace 
		\begin{aligned}
		 & \textbf{u}(\xi,t) = -b(T - t)^{-1 + \gamma}r\hat{r} + 2(T-t)^{-1 + \gamma}[bz + c(T-t)^{\gamma}]\hat{k}
		 \\ & \omega(\xi,t) = \vec{0}, 
		\end{aligned}  
\right.	
\end{equation}
for every $(\xi,t) \in \mathcal{C}(\delta,T)$, with $b, c \in \mathbb{R}$. Recall that the self-similar ansatz was proposed in order to describe a blow-up scenario at the space-time point $(r,z,t) = (1,0,T)$. Thus, the solution of problem \eqref{euler_alternativa1} must be smooth outside the region $\mathcal{C}(\delta,T)$, which enables us to deduce:
$$
\int_{0}^{T - \delta} \Vert \omega(\cdot,t) \Vert_{L^{\infty}(\Phi; \mathbb{R}^{3})} \,dt < \infty, \ \ \int_{T - \delta}^{T} \Vert \omega(\cdot,t) \Vert_{L^{\infty}(\Phi \setminus \Lambda; \mathbb{R}^{3})} \,dt < \infty.
$$
Since the vorticity field is identically null on $\Lambda \times (T - \delta, T)$, we have:
$$ 
\Vert \omega(\cdot,t) \Vert_{L^{\infty}(\Phi; \mathbb{R}^{3})} = \Vert \omega(\cdot,t) \Vert_{L^{\infty}(\Phi \setminus \Lambda; \mathbb{R}^{3})}, \ \ \forall t \in (T - \delta, T),
$$
and then:
$$ 
\int_{0}^{T} \Vert \omega(\cdot,t) \Vert_{L^{\infty}(\Phi; \mathbb{R}^{3})} \,dt = \int_{0}^{T-\delta} \Vert \omega(\cdot,t) \Vert_{L^{\infty}(\Phi; \mathbb{R}^{3})} \,dt + \int_{T-\delta}^{T} \Vert \omega(\cdot,t) \Vert_{L^{\infty}(\Phi \setminus \Lambda; \mathbb{R}^{3})} \,dt < \infty.
$$
Therefore, the BKM criterion assures that the solutions belonging to the family \eqref{solution a1} remain smooth during the whole interval of time $[0,T]$. In particular, they do not blow-up at time $T$ (even though the velocity field may exhibit a singularity at that time, depending on the value of $\gamma$).\\
\par
$\bullet$ \textbf{Case (B):} in cylindrical coordinates, the velocity and vorticity fields associated with the class of solutions \eqref{solution2} take the form:
\begin{equation} \label{solution b1}
\left\lbrace 
		\begin{aligned}
		 & \textbf{u}(\xi,t) = \kappa r (1 - r)^{\frac{\gamma - 2}{2\gamma}}\hat{\theta} + 2c(T-t)^{-1 + 2\gamma}\hat{k}
		 \\ & \omega(\xi,t) = \kappa\dfrac{2 - (\alpha + 2)r}{(1-r)^{\beta}}\hat{k}, 
		\end{aligned}  
\right.	
\end{equation}
for every $(\xi,t) \in \mathcal{C}(\delta,T)$, with $\kappa, c \in \mathbb{R}$, $\alpha \doteq \dfrac{1}{2} - \dfrac{1}{\gamma}$, $\beta \doteq \dfrac{1}{2} + \dfrac{1}{\gamma}$. Then, for $\kappa \neq 0$, the vorticity field is not identically null on $\mathcal{C}(\delta,T)$, but is governed by a stationary regime that depends only on the values of $r \in (1-\delta,1)$. Furthermore, for every $\kappa \neq 0$, we easily see that $\vert \omega_{z}(r) \vert \nearrow \infty$ as $r \nearrow 1$. In particular, we may deduce that:

$$ 
\Vert \omega(\cdot,t) \Vert_{L^{\infty}(\Phi; \mathbb{R}^{3})} = \sup_{1 - \delta < r < 1} \vert \omega _{z}(r) \vert = \infty, \ \ \forall t \in (T - \delta, T),
$$
and this clearly implies the following:
$$ 
\int\limits_{T-\delta}^{T-\delta + \varepsilon} \Vert \omega(\cdot,t) \Vert_{L^{\infty}(\Phi; \mathbb{R}^{3})} \,dt = \infty,
$$
for all $\varepsilon > 0$. Thus, according to the BKM criterion \eqref{bkm}, the solutions belonging to the family \eqref{solution b1} (with $\kappa \neq 0$) blow up at time $T - \delta + \varepsilon$, for every $\varepsilon > 0$ sufficiently small. Now, supposedly, the self-similar ansatz is valid in the entire region $\Lambda \times [T - \delta, T)$, but as we have seen, the turbulent regime is activated immediately at instant $T - \delta$, due to the stationary nature of the vorticity field. Therefore, the ansatz cannot be effective in the whole interval of time $[T - \delta, T)$. This is a clear discrepancy between the analytical properties of solution \eqref{solution b1} and the numerical evidence that was reported in \cite{luo}, but the main conclusion is this: functions belonging to family \eqref{solution b1} do not constitute solutions of the 3D axisymmetric Euler equations that are initially smooth (for some interval of time), but that develop a singularity at some later, finite time $T > 0$. On the contrary, such solutions are stationary and present a physical singularity at $r = 1$ that is independent of time and of the axial variable, even though the ansatz was designed to illustrate the development and evolution of a self-similar singularity occuring at the space-time point $(r,z,t) = (1,0,T)$.
\par
Another sign of inconsistency can be found at the end of section 4.7 (in \cite{luo}), where the authors outline the detection of the following scaling law for the axial component of the vorticity field (prior to the critical time $T$):
$$
\omega_{z} = \mathcal{O}(T - t)^{-2\text{.}45}.
$$
Certainly, such a scaling law cannot reproduce the behaviour of a stationary vorticity field.

\section{Final comments}
As mentioned earlier in the abstract, the family of solutions \eqref{solution1} was initially found by D. Chae and T.-P. Tsai in \cite{chae4}, under a rather unjustified decay condition of the blow-up profiles. Precisely, some calculations based on the uniform boundedness of the energy of the solution suggested them the following laws for the asymptotic behaviour of the self-similar profiles $(U,\Psi)$: 
\begin{equation} \label{gamma2}
\left\lbrace 
\begin{aligned}
& \vert U(R,Z) \vert = o(1), \ \ \text{for} \ \ 0 < \gamma < 2, \\
& \Vert \nabla \Psi (R,Z) \Vert = o(\Vert (R,Z) \Vert), \ \ \text{for} \ \gamma > 0,
\end{aligned}  
\right.	
\end{equation}
as $\Vert (R,Z) \Vert \longrightarrow \infty$, for both profiles. The hypothesis \eqref{gamma2} allowed them to derive the class of solutions \eqref{solution1} after a simple calculation, but at the same time, impeded the development of the second family of solutions \eqref{solution2}.  Therefore, the main novelty of our work lies in the obtainment of the general class of solutions \eqref{solution1} - \eqref{solution2} without assuming such asymptotic conditions on the self-similar profiles.
{\begin{center}
\textbf{Acknowledgements}
\end{center}
\par
This research was part of my thesis in mathematical engineering at the Universidad de Chile (Departamento de Ingenier\'{i}a Matem\'{a}tica), which was conducted under the supervision of professors Juan D\'{a}vila and Manuel del Pino (to whom I thank the many helpful comments), and supported partially by Fondecyt grant $1150066$.
\bibliographystyle{abbrv}
\bibliography{references}

\begin{thebibliography}{1}

\bibitem{beale}
J.~T. Beale, T.~Kato, and A.~Majda.
\newblock Remarks on the breakdown of smooth solutions for the 3{D} {E}uler
  equations.
\newblock {\em Communications in Mathematical Physics}, 94(1):61--66, 1984.

\bibitem{majda}
A.~L. Bertozzi and A.~J. Majda.
\newblock {\em Vorticity and Incompressible Flow}.
\newblock Cambridge University Press, 2002.

\bibitem{chae4}
D.~Chae and T.-P. Tsai.
\newblock Remark on {L}uo-{H}ou’s ansatz for a self-similar solution to the
  3{D} {E}uler equations.
\newblock {\em Journal of Nonlinear Science}, 25(1):193--202, 2015.

\bibitem{constantin2}
P.~Constantin.
\newblock Geometric statistics in turbulence.
\newblock {\em Siam Review}, 36(1):73--98, 1994.

\bibitem{ferrari}
A.~B. Ferrari.
\newblock On the blow-up of solutions of the 3{D} {E}uler equations in a
  bounded domain.
\newblock {\em Communications in Mathematical Physics}, 155(2):277--294, 1993.

\bibitem{gibbon}
J.~D. Gibbon.
\newblock The three-dimensional {E}uler equations: Where do we stand?
\newblock {\em Physica D: Nonlinear Phenomena}, 237(14):1894--1904, 2008.

\bibitem{luo}
T.~Hou and G.~Luo.
\newblock Potentially singular solutions of the 3{D} incompressible {E}uler
  equations.
\newblock {\em arXiv preprint arXiv:1310.0497}, 2013.

\bibitem{luo_res}
T.~Hou and G.~Luo.
\newblock Potentially singular solutions of the 3{D} axisymmetric {E}uler
  equations.
\newblock {\em Proceedings of the National Academy of Sciences},
  111(36):12968--12973, 2014.

\end{thebibliography}

\end{document}